\documentclass[11pt]{article}
\usepackage{amsmath}
\usepackage{amsfonts}
\usepackage{amsthm}
\usepackage{amssymb, setspace}
\usepackage{color,graphicx,epsfig,geometry,hyperref,fancyhdr}
\usepackage[T1]{fontenc}
\usepackage{float}
\usepackage{subfig}
\usepackage{bbm}
\usepackage{mathrsfs,fleqn}
\newtheorem{theorem}{Theorem}[section]

\newtheorem{lemma}{Lemma}[section]

\newcommand{\N}{\mathbb{N}}

\newcommand{\R}{\mathbb{R}}
\newcommand{\C}{\mathbb{C}}

\graphicspath{{pics/}}

\begin{document}

\begin{flushleft}
\Large 
\noindent{\bf \Large Regularized Factorization Method for a perturbed positive compact operator applied to inverse scattering}
\end{flushleft}

\vspace{0.2in}

{\bf  \large Isaac Harris}\\
\indent {\small Department of Mathematics, Purdue University, West Lafayette, IN 47907 }\\
\indent {\small Email: \texttt{harri814@purdue.edu}}\\


\begin{abstract}
\noindent In this paper, we consider a regularization strategy for the factorization method when there is noise added to the data operator.  The factorization method is a {\it qualitative} method used in shape reconstruction problems. These methods are advantageous to use due to the fact that they are computationally simple and require little a priori knowledge of the object one wishes to reconstruct. The main focus of this paper is to prove that the regularization strategy presented here produces stable reconstructions. We will show this is the case analytically and numerically for the inverse shape problem of recovering an isotropic scatterer with a conductive boundary condition. We also provide a strategy for picking the regularization parameter with respect to the noise level.  Numerical examples are given for a scatterer in 2 dimensions. 
\end{abstract}

\noindent {\bf Keywords}:  Factorization Method $\cdot$ Regularization $\cdot$ Shape Reconstruction   \\

\noindent {\bf MSC}:  35J05, 35Q81, 46C07

\section{Introduction}

We are interested, in studying a regularization strategy for the factorization method to prove that it is stable with respect to noise added to the positive compact data operator. This is a {\it qualitative} reconstruction method that can be used to solve many inverse shape problems. The factorization method was first introduced in \cite{firstFM} for reconstructing a sound soft or hard scatterer from the far-field measurements. Over the years the factorization method has become a useful analytical and computational tool for shape reconstruction. See the papers \cite{fm-waveguide,fm-eit-crack,fm-shixu,FM-wave,fm-gbc,GH1,fm-GR,Harris-Rome,FM-EIT,Liem} and the references therein for applications of the factorization method for solving inverse shape problems for elliptic and hyperbolic PDEs.

The main idea behind the factorization method is to connect the unknown region to be reconstructed with the range of your data operator. This is done by considering a linear ill-posed equation that is only solvable if and only if the `sampling point' is in the region of interest. Therefore, one can use Picard's criteria to reconstruct the region. To do so, one constructs an imaging functional that is a series where the sequence is defined by an inner--product in the numerator and the eigenvalues of a compact operator in the denominator. This could cause instabilities in  the reconstruction since the denominator tends to zero rapidly. To stabilize the numerical reconstructions the authors in \cite{GLSM} developed a generalized linear sampling method that uses the ideas from the factorization method to derive a new imaging functional. The analysis provided in \cite{arens,GLSM} connected the factorization method and the linear sampling method \cite{CK} (see \cite{rgf-colton,lsm-em,lsm-heat} for other applications). This idea was further studied in \cite{harris1,Harris-Rome} where a similar imaging functional was derived as in \cite{GLSM} using any suitable regularization scheme.

The imaging functional derived in the papers  \cite{harris1,Harris-Rome} are referred to as the regularized factorization method. Here we show that this method is stable with respect to noise in the data. The work in this paper is mainly influenced by the analysis in \cite{arens,arens2,Gebauer,FM-DoT-theory,RegFM}. These papers all study different imaging functionals from qualitative reconstruction methods to provided accurate and stable methods for shape reconstruction. In order to prove that the regularized factorization method is stable with respect to noise in the measured data, we will use results from perturbation theory to prove our main result.

The rest of the paper is structured as follows. We begin by discussing some results from perturbation theory that will be used in our analysis. First we discuss some known results and then we will provided the necessary extension to the problem under consideration. This will allow us to prove that  the regularized factorization method is stable with respect to noise added to the data. With this, we will then apply the theory to recover an isotropic scatterer with a conductive boundary. To do so, we will factorize the far-field operator and analyze the operators in the factorization to prove that our theory holds. Lastly, we will provide some numerical examples in  2 dimensions for recovering the scatterer. In our numerical experiments, we will derive an analytical method for picking the regularization parameter.

\section{Results from Perturbation Theory}\label{p-theory}
In this section, we will discuss some abstract results related to perturbation theory that will be used to prove the main result of the paper. The results that we will need pertain to the perturbation of a self-adjoint compact operator acting on a Hilbert space. We will review some of the results and analysis in \cite{p-theory-book,RegFM}. We are motivate by the work in \cite{harris1,Harris-Rome} where regularized variants of the factorization method were developed. This method has been applied to diffuse optical tomography \cite{harris1}, electrical impedance tomography \cite{GH1} and inverse scattering \cite{regfm-ams}. In the aforementioned papers, the results hold when one has the unperturbed data operator whereas we wish to extend the results when one only has access to the perturbed data operator. 

\subsection{Theory for positive self-adjoint compact operators}\label{standard-theory}
To begin, we will assume that $K$ and $K^\delta: X \longrightarrow X$ are a pair of positive self-adjoint compact operators acting on a Hilbert space $X$. We will also assume that, $K^\delta$ is a perturbation of the operator $K$ such that $\| K-K^\delta\| \leq \delta$ { for some } $ 0 < \delta \ll 1$
where $\| \cdot \|$ denotes the operator norm. From the Hilbert-Schmidt Theorem, we have that both operators are orthogonally diagonalizable such that 
$$K x = \sum\limits_{j=1}^{\infty} \lambda_{n} (x,x_n)_X x_n \quad \text{ and  } \quad K^\delta x = \sum\limits_{n=1}^{\infty} \lambda^\delta_{n} \big(x,x^\delta_n\big)_X x^\delta_n$$
where $\lambda_n$ and  $\lambda^\delta_n \in \R_{>0}$ are the eigenvalues in non-increasing  order that tend to zero as $n \to \infty$. Here, $x_n$ and $x^\delta_n$ are the corresponding eigenfunctions that form an orthonormal basis of $X$.  

We have the continuity of the spectrum i.e. Hausdorff distance between the spectrums satisfies that 
$$\text{dist} \big(\text{spec}(K), \text{spec}(K^\delta) \big)\leq \| K-K^\delta\| $$ 
where we let $\text{spec}(\cdot )$ denote the set of eigenvalues for a self-adjoint compact operator \cite{p-theory-book}. Now, assume that for a fixed $n \in \N$ we have that 
$$ \text{dist}\big(\lambda_n , \text{spec}(K)\setminus \{\lambda_n\} \big) =\inf \big\{ | \lambda_n - \lambda_m |  \,\,: \,\, \lambda_n \neq \lambda_m   \,\,\,\, \text{with $\lambda_m \in$ spec$(K)$}   \big\}  \geq \rho$$
for some $\rho>0$. Then, we can define  the spectral projection as in \cite{p-theory-book} on the eigenspace corresponding to the eigenvalue $\lambda_n$ and $\lambda^\delta_n$ which are given by the 
\begin{align} \label{projection}
P_{n} = \frac{1}{2\pi \text{i}}\int_{\Gamma_n} (\lambda I-K)^{-1} \, \text{d}\lambda  \quad \text{ and  } \quad P^\delta_{n} =  \frac{1}{2\pi \text{i}} \int_{\Gamma_n} \big(\lambda I-K^\delta\big)^{-1} \, \text{d}\lambda 
\end{align}
where $\Gamma_n = \partial B(\lambda_n ; \rho/2)$ provided that $\rho/2>\delta$. Here, we define the sets for the integrals as 
$$\partial B(\lambda_n ; \rho/2)  = \big\{ \xi \in \C \,\,: \,\, |\xi -\lambda_n| =  \rho/2 \,\,  \,\,  \text{for some given $\rho>0$} \}.$$
Therefore, the integrals are over the contour $\Gamma_n$ in the complex plane. Notice, that since we have assumed that $\rho/2>\delta$ this implies that the intersection of $\Gamma_n$ with either $\text{spec}(K)$  or $\text{spec}(K^\delta)$ is empty. Indeed, since we have assumed that $\text{dist}\big(\lambda_n , \text{spec}(K)\setminus \{\lambda_n\} \big) \geq \rho$ we have that 
$$\text{dist}\big(\Gamma_n , \text{spec}(K) \big) \geq \rho/2$$ 
and by the triangle inequality we can easily obtain that 
$$\text{dist}\big(\Gamma_n , \text{spec}(K^\delta) \big) \geq \rho/2 - \delta \quad \text{which is assumed to be positive.}$$
Therefore, the contour integrals in \eqref{projection} are well defined bounded linear operators by the Fredholm Alternative (see for e.g. \cite{p-theory-book,RegFM}). 

By Theorem 4.2 in \cite{RegFM} we have the following norm estimate
\begin{align} \label{p-estimate1}
\big\| P_{n} - P^\delta_{n}  \big\| \leq \frac{\delta}{\rho/2-\delta} \quad \text{ provided that  } \quad \rho/2>\delta. 
\end{align}
Since, we are interested in the case when $0 < \delta \ll 1$ we will assume that $\delta \in (0,1/4)$ which gives that we can take $\rho/2 = \sqrt{\delta}$. With this, some simple calculations using \eqref{p-estimate1} gives that 
\begin{align} \label{p-estimate2}
\big\| P_{n} - P^\delta_{n}  \big\| \leq 2 \sqrt{\delta} \quad \text{ provided that  } \quad \delta \in (0,1/4). 
\end{align}
Now, by Proposition 4.3 of \cite{RegFM} we have that the projection operators are given by 
$$P_{n} x = \sum_{\lambda = \lambda_n} (x,x_n)_X x_n \quad \text{and} \quad P^\delta_{n} x = \sum_{\lambda = \lambda^\delta_n} \big(x,x^\delta_n\big)_X x^\delta_n$$
i.e. the projection  onto the space spanned by the orthonormal eigenfunctions corresponding to a specific eigenvalue.
Therefore, we have that 
\begin{align} \label{p-estimate3}
\sum_{\lambda = \lambda^\delta_n} |(x,x^\delta_n)_X|^2 - \sum_{\lambda = \lambda_n} \big| \big(x,x_n\big)_X \big|^2 = \| P^\delta_n x\|_X^2 -\| P_n x \|^2_X \leq 4 \| x\|^2 \big\| P^\delta_{n} - P_{n}  \big\| .
\end{align}
The above estimate is obtained by using the definition of the norm on $X$ and the triangle inequality (see \cite{RegFM} for details). With this, we can now extend these result for a positive operator mapping a Hilbert Space into it's dual space. 

\subsection{Extension of standard perturbation results}\label{new-theory}

In this section, we will use the perturbation theory discussed above for positive self-adjoint compact operators to positive compact operators that map $X$ into it's dual space $X^*$. To this end, we assume that $A$ and $A^\delta: X \to X^*$ are acting on the complex Hilbert space $X$ that are positive and compact. As in the previous section, we will assume that $A^\delta$ is a perturbation of $A$ satisfying the inequality 
\begin{align}\label{p-bound1}
\| A-A^\delta\| \leq \delta \quad \text{ for some } \quad 0 < \delta \ll 1.
\end{align}
Here, we will assume that $\langle \cdot \, , \cdot \rangle_{X\times X^*}$ denote the sesquilinear dual-pairing between $X$ and $X^*$.  Furthermore, we shall assume that $H$ is the Hilbert pivoting space such that the dual-pairing coincides with the inner-product on the $H$ with $X \subset H \subset X^*$ (with dense inclusion) forming a Gelfand triple.

In order to use the theory for self-adjoint compact operators, we let $R : X^* \to X$ denote the bijective isometry given by the Riesz Representation Theorem such that 
\begin{align} \label{RRT}
R \ell = x_\ell \quad \text{ where } \quad  (x \, , x_\ell )_X = \langle x \, , \ell \rangle_{X\times X^*}  \quad \text{ for all } \quad  x \in X.
\end{align}
Note, that due to the fact that the dual-pairing is sesquilinear, we have that $R$ is a linear isometry. Therefore, we have that $RA$ and $RA^\delta: X \to X$ and satisfy 
$$\| RA-RA^\delta\| = \| A-A^\delta\| \leq \delta.$$
Notice, that the operator $R A : X \to X$ satisfies 
$$\big(x \, , (RA) x \big)_X = \langle x \, , A x \rangle_{X\times X^*} > 0 \quad \text{ for all } \quad  x \in X \setminus \{0\}$$
since $A$ is assumed to be positive and similarly for $RA^\delta$. By appealing to Corollary 7.3 of \cite{LA-done} (see also Theorem 3:10-3 in \cite{func-analysis}) we have that $RA$ and $RA^\delta$ are positive self-adjoint compact operators acting on the complex Hilbert Space $X$. This implies that, we have the results and estimates from Section \ref{standard-theory} where $K=RA$ and $K^\delta=RA^\delta$. From this, we let 
$$\{\lambda_n ; x_n \} \in \R_{>0}\times X \quad \text{ and } \quad \{\lambda^\delta_n ; x^\delta_n \} \in \R_{>0}\times X$$
denote the eigenvalues and orthonormal functions for $RA$ and $RA^\delta$, respectively. By the continuity of the spectrum we have that 
$$\text{dist}\big(\text{spec}(RA), \text{spec}(RA^\delta) \big) \leq \delta.$$ 
Now, we can define the corresponding orthonormal dual-basis $\ell_n$ and  $\ell^\delta_n \in X^*$ such that 
$$R \ell_n = x_n \quad \text{and } \quad R \ell^\delta_n = x^\delta_n  \quad \text{for all } \,\,\, n \in \N, \quad  \text{ respectively.}$$ 
Note, that  $X^*$ is also a Hilbert space with the inner--product 
$$ (\ell , \varphi )_{X^*} = (x_\ell , x_\varphi )_{X} \quad \text{ for all } \quad \ell, \varphi \in X^*  \quad \textrm{ where }\;  R \ell = x_\ell \; \textrm{ and } \; R \varphi = x_\varphi.$$
From the analysis in \cite{harris1}, we have that 
$$\{\lambda_n ; x_n ; \ell_n \} \in \R_{>0}\times X\times X^*  \quad \text{ and } \quad  \{\lambda^\delta_n ; x^\delta_n ; \ell^\delta_n \} \in \R_{>0}\times X\times X^*$$ 
corresponds to the singular value decomposition for the operators $A$ and $A^\delta$, respectively. 

With this we can now provide the main perturbation result that will be used to study the regularized factorization method in the preceding section. To this end, following in a similar manner as in \cite{RegFM} we need to define 
\begin{align}\label{n-delta}
N(\delta) = \sup \Big\{ n \in \N \,\, : \,\, \text{dist}\big(\lambda_n , \text{spec}(RA)\setminus \{\lambda_n\} \big) \geq 2 \sqrt{\delta} \,\,  \textrm{ and }  \,\, 8n\sqrt[4]{\delta} \leq 1 \Big\}.
\end{align}  
Notice, that as $\delta \to 0^+$ we have that $N(\delta) \to \infty$. Now, we prove a vital result for extending the regularized factorization method for a perturbed positive compact operators mapping the Hilbert space $X$ into the dual space. 

\begin{theorem}\label{p-conv}
Assume that $A$ and $A^\delta: X \to X^*$ are positive and compact satisfy \eqref{p-bound1}. Then for any $n \in \N$ we have that $\lambda^\delta_n \to \lambda_n$ as $\delta \to 0^+$ as well as 
$$\sum_{n=1}^{N(\delta)} \Big[ | \langle x^\delta_n,\ell  \rangle_{X\times X^*}|^2 - |\langle x_n,\ell  \rangle_{X\times X^*}|^2 \Big] \leq \sqrt[4]{\delta} \| \ell \|_{X^*}^2$$
where $N(\delta)$ is defined by \eqref{n-delta} provided that $\delta \in (0,1/4)$ for any $\ell \in X^*$. 
\end{theorem}
\begin{proof}
To begin the proof, notice that by the continuity of the spectrum for each $n \in \N$ we have the estimate  $|\lambda^\delta_n - \lambda_n | \leq \delta$ proving the convergence of the eigenvalues.

To prove the claimed estimate, we first note that by \eqref{RRT} we have that 
$$ \langle x_n , \ell \rangle_{X\times X^*} = ( x_n  , x_\ell )_X  \quad \text{ and } \quad \langle x^\delta_n ,\ell \rangle_{X\times X^*} = ( x^\delta_n  , x_\ell )_X \quad \text{ for all } \quad n \in \N$$ 
with $x_\ell = R \ell$. Therefore, we have that 
$$ \sum_{n=1}^{N(\delta)} \Big[ | \langle x^\delta_n,\ell  \rangle_{X\times X^*}|^2 - |\langle x_n,\ell  \rangle_{X\times X^*}|^2 \Big] = \sum_{n=1}^{N(\delta)} \Big[ |(x_\ell \, , x^\delta_n )_X |^2 -|(x_\ell \, , x_n )_X |^2 \Big].$$ 
We now, let $M(\delta)$ be the number of distinct eigenvalues from $\lambda_1\geq \cdots \geq \lambda_{N(\delta)}$. With this, we can appeal to \eqref{p-estimate3} to obtain that 
$$\sum_{n=1}^{N(\delta)} \Big[ | \langle x^\delta_n,\ell  \rangle_{X\times X^*}|^2 - |\langle x_n,\ell  \rangle_{X\times X^*}|^2 \Big]  = \sum_{n=1}^{M(\delta)} \| P^\delta_n x_\ell \|_X^2 -\| P_n x_\ell \|^2_X$$
where $P_{n}$ and  $P^\delta_{n}$ are given by \eqref{projection} with the self-adjoint compact operators $RA$ and $RA^\delta$, respectively. 
We can again use \eqref{p-estimate3} to obtain the estimate 
\begin{align*}
\sum_{n=1}^{N(\delta)}\Big[ | \langle x^\delta_n,\ell  \rangle_{X\times X^*}|^2 &- |\langle x_n,\ell  \rangle_{X\times X^*}|^2 \Big] \\
& \leq M(\delta) \max \Big\{ \| P^\delta_n x_\ell \|_X^2 -\| P_n x_\ell \|^2_X  \,\, : \,\, n \leq N(\delta) \Big\} \\
& \leq   N(\delta) \max \Big\{ 4 \| x_\ell \|_X^2 \big\| P_{n} - P^\delta_{n}  \big\|  \,\, : \,\, n \leq N(\delta) \Big\},
\end{align*}
where we have used the fact that $M(\delta) \leq N(\delta)$. Now, by \eqref{p-estimate2} we have that 
$$\max\Big\{ 4 \| x_\ell \|_X^2 \big\| P_{n} - P^\delta_{n}  \big\|  \,\, : \,\, n \leq N(\delta) \Big\} \leq 8 \sqrt{\delta} \| x_\ell \|_X^2 $$
since we have assumed that $\delta \in (0,1/4)$. By the fact that $R$ is an isometry, we have the equality  $\| x_\ell \|_X = \| \ell \|_{X^*}$. Combining the above inequalities gives that 
\begin{align*}
\sum_{n=1}^{N(\delta)}\Big[ | \langle x^\delta_n,\ell  \rangle_{X\times X^*}|^2 - |\langle x_n,\ell  \rangle_{X\times X^*}|^2 \Big]  & \leq 8 N(\delta) \sqrt{\delta}  \| \ell \|^2_{X^*}\\
&\leq  \sqrt[4]{\delta} \| \ell \|^2_{X^*}
\end{align*}
by using the fact that $8N(\delta)\sqrt[4]{\delta} \leq 1$, proving the claim. 
\end{proof} 

Notice, that the a rephrased version of Theorem \ref{p-conv} is still valid for the case when the operators map the $X$ into itself. Using this result we can prove that the regularized factorization method is stable with respect to noisy data. Also, we can remove the assumption that $X$ is a complex Hilbert Space by adding the assumption that the mapping 
$$(x,y) \longmapsto  \langle y \, , A x \rangle_{X\times X^*} $$
is symmetric see \cite{harris1} for details.

\section{Regularized Factorization Method with Error}\label{new-RegFM}
In this section, we will study the regularized factorization method for a perturbed positive operator $A^\delta: X \to X^*$. As in the previous section, we assume that $X$ is an infinite dimensional  Hilbert space and $X^*$ is the corresponding dual space where $X \subseteq H \subseteq X^*$ forming a Gelfand triple with Hilbert pivoting space $H$. In our analysis, we will assume that $A^\delta$ is a perturbation of the operator $A: X \to X^*$ satisfying \eqref{p-bound1}. The operator $A$, is assumed to have the factorization
\begin{align}\label{factorization}
A=S^* T S  \quad \text{ with } \quad S: X \to V  \quad \text{ and } \quad T: V \to V
\end{align}
with $V$ also being a Hilbert space. The adjoint operator for $S$ is the mapping $S^*: V \to X^*$ satisfying the equality 
\begin{align*}
(S x , v )_{V} = \langle x , S^* v \rangle_{X\times X^*} \quad \text{ for all } \quad  v \in {V} \textrm{ and } x \in X.
\end{align*}
Furthermore, We will assume that the operator $S$ is compact and injective where as  the operator $T$ is bounded and strictly coercive on Range$(S)$. Notice that,  from the factorization of the operator $A$ we have that it is also positive and compact.

In \cite{harris1}, it is proven that one can connect the $\text{Range} (S^*)$ to the singular value decomposition of $A$ denoted $\{\lambda_n ; x_n ; \ell_n \} \in \R_{>0}\times X\times X^* $ such that 
\begin{align}\label{regfm-thm1}
\ell \in \text{Range}(S^* ) \quad \text{if and only if} \quad  \sum \frac{1}{\lambda_n}  \left|\langle x_n ,\ell \rangle_{X\times X^*} \right|^2 <\infty
\end{align}
(see \cite{Gebauer} for the case when $A$ maps $X$ into itself). Note, that due to the fact that $\lambda_n \to 0$ (rapidly) as $n \to \infty$ a regularized version of \eqref{regfm-thm1} was proven. This is due to the fact that, in shape reconstruction problems, using  \eqref{regfm-thm1} could result in some numerical instabilities (see for e.g. \cite{regfm-ams}). With the above assumptions, it is shown that 
\begin{align}\label{regfm-thm2}
\ell \in\text{Range} (S^*) \quad \text{ if and only if} \quad \liminf\limits_{\alpha \to 0} \langle x^\alpha \, , A x^\alpha \rangle_{X\times X^*} < \infty
\end{align}
where $x^\alpha$ is the regularized solution to $Ax=\ell$. In order to define the regularized solution $x^\alpha$ we again use the singular value decomposition of $A$ which gives that 
\begin{align*}
x^\alpha =  \sum \frac{\phi_\alpha(\lambda_n)}{\lambda_n} \overline{ \langle x_n,\ell  \rangle}_{X\times X^*} \, x_n
\end{align*}
where we have used that 
$$( \ell_n \, , \ell )_{X^*} = \langle x_n \, , \ell \rangle_{X\times X^*}  \quad \text{ for all } \quad  n \in \N \quad \text{ where  } \quad R \ell_n =x_n.$$
Here, we will assume that for $\alpha>0$ the family of filter functions $\phi_\alpha (t) : \big[0 ,\|A\| \big] \to \R_{\geq0}$ satisfies that for $0 < t \leq  \|A\|$
\begin{align} \label{filter-assumptions}
\lim\limits_{\alpha \to 0} \phi_\alpha(t) =1, \quad \phi_\alpha(t) \leq C_{\text{reg}}  \quad \text{ and } \quad  \phi_\alpha(t) \leq C_{\alpha} t \quad \text{for all} \,\, \alpha>0
\end{align}
where the constant  $C_{\text{reg}}$ is independent of the regularization parameter $\alpha$. The filter functions for Tikhonov regularization and Landweber iteration are given by 
\begin{align} \label{filters}
\phi_\alpha (t) =  \frac{t^2}{t^2+\alpha} \quad \text{and} \quad \phi_\alpha(t) = 1-\left( 1 - \beta t^2\right)^{1/\alpha},
 \end{align}
respectively (see for e.g. \cite{kirschipbook}). For the Landweber iteration we assume that $\alpha=1/m$ for some $m \in \N$ and constant $\beta < 1/  \|A\|^2$. Note, that the assumptions on the filter functions in equation \eqref{filter-assumptions} are standard in regularization theory. 

Next, we will prove a similar result as in \eqref{regfm-thm2} where one uses the perturbed operator $A^\delta$. This is usually the case in applications where the measurements are polluted by random noise.  One last assumption we need is that for all 
\begin{align}\label{filter-cont}
\phi_\alpha (\lambda_n^\delta)  \longrightarrow  \phi_\alpha (\lambda_n)  \quad \text{as} \quad \delta \longrightarrow  0^+ \quad \text{for all} \quad \alpha>0.
\end{align}
Therefore, we will assume that $\phi_\alpha (t)$ is continuous with respect to $0 \leq t \leq  \|A\|$. This is true for the filter functions presented in \eqref{filters}. It is clear, that for both filter functions in \eqref{filters} we have that $C_{\text{reg}}=1$ along with 
$$ C_\alpha = 1/(2\sqrt{\alpha}) \quad \text{and } \quad  C_\alpha = \sqrt{\beta/\alpha}$$
for Tikhonov regularization and Landweber iteration, respectively (see for e.g. Theorem 2.8 of \cite{kirschipbook}). Notice, that the condition $\phi_\alpha(t) \leq C_{\alpha} t$ implies that the mapping $t \mapsto {\phi^2_\alpha(t)}/{t}$ for $t>0$ and ${\phi^2_\alpha(t)}/{t}=0$ at $t=0$ is uniformly continuous on $\big[0 ,\|A\| \big]$.

From this, we note that now using the singular value decomposition for $A^\delta$ denoted by $\{\lambda^\delta_n ; x^\delta_n ; \ell^\delta_n \} \in \R_{>0}\times X\times X^*$ then we have that 
\begin{align}\label{regsolu}
x^{\delta,\alpha} =  \sum \frac{\phi_\alpha(\lambda^\delta_n)}{\lambda^\delta_n} \overline{ \langle x^\delta_n,\ell  \rangle}_{X\times X^*} \, x^\delta_n 
\end{align}
where $x^{\delta,\alpha}$ is the regularized solution to $A^\delta x=\ell$. Notice that, in \eqref{regsolu} we have used the fact that 
$$(\ell \, , \ell^\delta_n )_{X^*} = \langle x_\ell \, , \ell^\delta_n \rangle_{X\times X^*}  \quad \text{ for all } \quad  n \in \N  \quad \text{ where  } \quad R \ell =x_\ell.$$
With the expression for $x^{\delta,\alpha}$ given in \eqref{regsolu}, we are now ready to prove the main result of this section i.e. to extend the result in equation \eqref{regfm-thm2} for perturbed operator $A^\delta$.

\begin{theorem}\label{regfm-thm3}
Let $A^\delta: X \to X^*$ be a positive compact operator that is the perturbation of the operator $A: X \to X^*$ satisfying \eqref{p-bound1}. Assume that $A=S^* T S$ such that $S$ is compact and injective as well as $T$ being strictly coercive on Range$(S)$.  Then we have that
$$\ell \in\text{Range} (S^*) \iff \liminf\limits_{\alpha \to 0^+} \liminf\limits_{\delta \to 0^+} \langle x^{\delta,\alpha} \, , A^\delta x^{\delta,\alpha} \rangle_{X\times X^*} < \infty$$
where $x^{\delta,\alpha}$ is the regularized solution given by \eqref{regsolu} to $A^\delta x=\ell$. 
\end{theorem} 
\begin{proof}
Notice, that do to the fact that ${\displaystyle \{\lambda^\delta_n ; x^\delta_n ; \ell^\delta_n \} \in \R_{>0}\times X\times X^* }$ is the singular value decomposition for the compact operator $A^\delta$ we have that $A^\delta x^\delta_n = \lambda^\delta_n \ell^\delta_n$ for any $n \in \N$. Therefore, we see that 
$$ A^\delta x^{\delta,\alpha} =  \sum {\phi_\alpha(\lambda^\delta_n)} \overline{ \langle x^\delta_n,\ell  \rangle}_{X\times X^*} \, \ell^\delta_n$$
by appealing to \eqref{regsolu}. From the fact that, $\ell^\delta_n$ is the dual basis for $x^\delta_n$ with respect to the dual-paring, 
we obtain the equality 
$$\langle x^{\delta,\alpha} \, , A^\delta x^{\delta,\alpha} \rangle_{X\times X^*} = \sum \frac{\phi^2_\alpha(\lambda^\delta_n)}{\lambda^\delta_n} | \langle x^\delta_n,\ell  \rangle_{X\times X^*}|^2.$$
In a similar manner we have that 
$$\langle x^\alpha \, , A x^\alpha \rangle_{X\times X^*} = \sum \frac{\phi^2_\alpha(\lambda_n)}{\lambda_n} | \langle x_n,\ell  \rangle_{X\times X^*}|^2.$$
In order to prove the claim, we bound (above and below) the quantity 
$$\liminf\limits_{\alpha \to 0^+} \liminf\limits_{\delta \to 0^+} \langle x^{\delta,\alpha} \, , A^\delta x^{\delta,\alpha} \rangle_{X\times X^*} $$
by the quantity $\liminf\limits_{\alpha \to 0^+} \langle x^\alpha \, , A x^\alpha \rangle_{X\times X^*}$ and apply the result in equation \eqref{regfm-thm2}. 

To this end, we will now prove the aforementioned upper bound. Therefore, we assume that $N(\delta)$ is defined by \eqref{n-delta} then we have that 
\begin{align}\label{thm-refequ1}  
\langle x^{\delta,\alpha} \, , A^\delta x^{\delta,\alpha} \rangle_{X\times X^*} &= \sum \frac{\phi^2_\alpha(\lambda^\delta_n)}{\lambda^\delta_n} | \langle x^\delta_n,\ell  \rangle_{X\times X^*}|^2  \\ \nonumber
&\hspace{-0.5in}= \sum\limits_{n=1}^{N(\delta)} \frac{\phi^2_\alpha(\lambda^\delta_n)}{\lambda^\delta_n} | \langle x_n,\ell  \rangle_{X\times X^*}|^2 + \sum\limits_{n=1}^{N(\delta)} \frac{\phi^2_\alpha(\lambda^\delta_n)}{\lambda^\delta_n}  \Big[ | \langle x^\delta_n,\ell  \rangle_{X\times X^*}|^2 - | \langle x_n,\ell  \rangle_{X\times X^*}|^2 \Big]  \\ &\hspace{1in}+ \sum\limits_{n=N(\delta)+1}^\infty \frac{\phi^2_\alpha(\lambda^\delta_n)}{\lambda^\delta_n} | \langle x^\delta_n,\ell  \rangle_{X\times X^*}|^2. \nonumber
\end{align}
Notice that, since  $A$ has dense range in $X^*$, this implies that $A$ has infinitely many distinct eigenvalues (since $X$ is infinite dimensional) and therefore  $N(\delta)$ tends to infinity as $\delta \to 0^+$. 

To prove the required upper bound, we first consider the middle term in the second line of equation \eqref{thm-refequ1}
\begin{align*}
\sum\limits_{n=1}^{N(\delta)} \frac{\phi^2_\alpha(\lambda^\delta_n)}{\lambda^\delta_n}  \Big[ | \langle x^\delta_n,\ell  \rangle_{X\times X^*}|^2  &- |\langle x_n,\ell  \rangle_{X\times X^*}|^2 \Big] \\
&\leq C^2_\alpha \sum\limits_{n=1}^{N(\delta)} \lambda^\delta_n  \Big[ | \langle x^\delta_n,\ell  \rangle_{X\times X^*}|^2 - |\langle x_n,\ell  \rangle_{X\times X^*}|^2 \Big] \\
&\leq C^2_\alpha \lambda^\delta_1 \sum\limits_{n=1}^{N(\delta)} \Big[ | \langle x^\delta_n,\ell  \rangle_{X\times X^*}|^2 - |\langle x_n,\ell  \rangle_{X\times X^*}|^2 \Big]. 
\end{align*}
Notice that,  we have used the assumptions on the filter functions in \eqref{filter-assumptions} in the first inequality and the fact that the singular values $\lambda^\delta_n$ are assumed to be in non-increasing order. By appealing to Theorem \ref{p-conv} we have that 
$$\sum\limits_{n=1}^{N(\delta)} \frac{\phi^2_\alpha(\lambda^\delta_n)}{\lambda^\delta_n}  \Big[ | \langle x^\delta_n,\ell  \rangle_{X\times X^*}|^2  - |\langle x_n,\ell  \rangle_{X\times X^*}|^2 \Big]  \leq  C^2_\alpha  \, \lambda^\delta_1 \,  \sqrt[4]{\delta}  \, \| \ell \|^2_{X^*}.$$ 
Now, we will estimate the last term in \eqref{thm-refequ1} such that 
\begin{align*}
\sum\limits_{n=N(\delta)+1}^\infty \frac{\phi^2_\alpha(\lambda^\delta_n)}{\lambda^\delta_n} | \langle x^\delta_n,\ell  \rangle_{X\times X^*}|^2 & \leq C^2_\alpha \sum\limits_{n=N(\delta)+1}^\infty {\lambda^\delta_n}  | \langle x^\delta_n,\ell  \rangle_{X\times X^*}|^2\\
&\leq C^2_\alpha  {\lambda^\delta_{N(\delta)+1} } \sum\limits_{n=N(\delta)+1}^\infty  | \langle x^\delta_n,\ell  \rangle_{X\times X^*}|^2 \\
& \leq C^2_\alpha  {\lambda^\delta_{N(\delta)+1} } \| \ell \|_{X^*}^2.
\end{align*}
Here, we have used the fact that $( \ell^\delta_n \, , \ell )_{X^*} = \langle x^\delta_n \, , \ell \rangle_{X\times X^*}$ for all  $n \in \N$
and the fact  that $\ell^\delta_n$ is an orthonormal sequence in  $X^*$. Combining these two estimates with \eqref{thm-refequ1}, we have that 
$$\langle x^{\delta,\alpha} \, , A^\delta x^{\delta,\alpha} \rangle_{X\times X^*} \leq   \sum\limits_{n=1}^{N(\delta)} \frac{\phi^2_\alpha(\lambda^\delta_n)}{\lambda^\delta_n} | \langle x_n,\ell  \rangle_{X\times X^*}|^2  + C^2_\alpha  {\lambda^\delta_{N(\delta)+1} } \| \ell \|_{X^*}^2 +  C^2_\alpha \lambda^\delta_1 \sqrt[4]{\delta} \| \ell \|^2_{X^*}.$$
We see that by taking the  $\liminf\limits_{\delta \to 0^+}$ of the above inequality, we have the estimate 
$$\liminf\limits_{\delta \to 0^+} \langle x^{\delta,\alpha} \, , A^\delta x^{\delta,\alpha} \rangle_{X\times X^*} \leq   \sum \frac{\phi^2_\alpha(\lambda_n)}{\lambda_n} | \langle x_n,\ell  \rangle_{X\times X^*}|^2. $$
This is obtained by showing the the first term converges to the desired estimate by standard arguments where as the other terms tend to zero. With this, we have the upper bound 
\begin{align}\label{upper}
\liminf\limits_{\alpha \to 0^+} \liminf\limits_{\delta \to 0^+} \langle x^{\delta,\alpha} \, , A^\delta x^{\delta,\alpha} \rangle_{X\times X^*} \leq \liminf\limits_{\alpha \to 0^+} \langle x^\alpha \, , A x^\alpha \rangle_{X\times X^*}.
\end{align}

Now, we prove a similar lower bound to complete the proof. Therefore, we again need to estimate

\begin{align*}
\langle x^{\delta,\alpha} \, , A^\delta x^{\delta,\alpha} \rangle_{X\times X^*} &= \sum \frac{\phi^2_\alpha(\lambda^\delta_n)}{\lambda^\delta_n} | \langle x^\delta_n,\ell  \rangle_{X\times X^*}|^2  \\ 
&\geq \sum\limits_{n=1}^{N(\delta)}  \frac{\phi^2_\alpha(\lambda^\delta_n)}{\lambda^\delta_n} | \langle x^\delta_n,\ell  \rangle_{X\times X^*}|^2\\ 
&\hspace{-0.4in}= \sum\limits_{n=1}^{N(\delta)} \frac{\phi^2_\alpha(\lambda^\delta_n)}{\lambda^\delta_n} | \langle x_n,\ell  \rangle_{X\times X^*}|^2 + \sum\limits_{n=1}^{N(\delta)} \frac{\phi^2_\alpha(\lambda^\delta_n)}{\lambda^\delta_n}  \Big[ | \langle x^\delta_n,\ell  \rangle_{X\times X^*}|^2 - | \langle x_n,\ell  \rangle_{X\times X^*}|^2 \Big] 
\end{align*}
From the previous estimates, we have that  
\begin{align}\label{reg-ineq}
\langle x^{\delta,\alpha} \, , A^\delta x^{\delta,\alpha} \rangle_{X\times X^*} \geq  \sum\limits_{n=1}^{N(\delta)} \frac{\phi^2_\alpha(\lambda^\delta_n)}{\lambda^\delta_n} | \langle x_n,\ell  \rangle_{X\times X^*}|^2 - C^2_\alpha \lambda^\delta_1 \sqrt[4]{\delta} \| \ell \|^2_{X^*}. 
\end{align}
Again, we take the  $\liminf\limits_{\delta \to 0^+}$ of the above inequality \eqref{reg-ineq} to obtain that 
\begin{align*}
\liminf\limits_{\delta \to 0^+}\langle x^{\delta,\alpha} \, , A^\delta x^{\delta,\alpha} \rangle_{X\times X^*} \geq  \sum \frac{\phi^2_\alpha(\lambda_n)}{\lambda_n} | \langle x_n,\ell  \rangle_{X\times X^*}|^2  
\end{align*}
 where we have again used that $N(\delta) \to \infty$ as $\delta \to 0^+$ as well as the continuity of the spectrum. Therefore, just as in proving the upper bound we take  $\liminf\limits_{\alpha \to 0^+}$ to obtain
\begin{align}
\liminf\limits_{\alpha \to 0^+}\liminf\limits_{\delta \to 0^+}\langle x^{\delta,\alpha} \, , A^\delta x^{\delta,\alpha} \rangle_{X\times X^*} \geq \liminf\limits_{\alpha \to 0^+}\langle x^\alpha \, , Ax^\alpha \rangle_{X\times X^*}. \label{lower}
\end{align}
Combining the estimates in equations \eqref{upper} and \eqref{lower}, we have that 
$$\liminf\limits_{\alpha \to 0^+}\langle x^\alpha \, , Ax^\alpha \rangle_{X\times X^*} < \infty \iff   \liminf\limits_{\alpha \to 0^+}\liminf\limits_{\delta \to 0^+}\langle x^{\delta,\alpha} \, , A^\delta x^{\delta,\alpha} \rangle_{X\times X^*} < \infty$$
and the result follows directly from equation \eqref{regfm-thm2}, proving the claim.
\end{proof}

We see that the equation \eqref{regfm-thm2} and the newly obtained result in Theorem \ref{regfm-thm3} are similar to the results found in \cite{GLSM} (see also \cite{GLSM2}). In \cite{GLSM}, the authors developed the Generalized Linear Sampling Method (GLSM). In short, the GLSM considers minimizing the functional 
$$\mathcal{J}_{\alpha} \big(x ; \ell \big)  = \alpha   \langle x \, , Ax \rangle_{X\times X^*}  +\| Ax - \ell  \|^2_{X^*}$$
where $A$ has the factorization \eqref{factorization} (under less restrictions on $T$). For this case, it can be shown that the minimizer of the functional is given by 
$$x^\alpha=\sum \frac{\lambda_n}{\alpha \lambda_n +\lambda_n^2}  \overline{ \langle x_n ,\ell \rangle}_{X\times X^*} x_n.$$
Notice, this imply that the filter function corresponding to the GLSM is given by 
\begin{align}\label{GLSMfilter}
\phi_\alpha (t) =  \frac{t}{\alpha    + t}
\end{align}
and notice that for all $t>0$
$$\lim\limits_{\alpha \to 0} \phi_\alpha(t) =1, \quad   \phi_\alpha(t) \leq 1 \quad \text{ and } \quad \phi_\alpha(t) \leq C_\alpha t \quad \text{for all} \quad \alpha>0$$
where $C_\alpha =1/\alpha$. 
Therefore, we can see that the GLSM for the perturbed operator $A^\delta$ fits with in the theory presented here. From this, provided that $x^{\delta,\alpha}$ is the minimizer of 
$$\mathcal{J}^\delta_{\alpha} \big(x ; \ell \big)  = \alpha   \langle x \, , A^\delta x \rangle_{X\times X^*}  +\| A^\delta x - \ell  \|^2_{X^*}$$
where $A$ and $A^\delta$ satisfy the assumptions of Theorem \ref{regfm-thm3} we can conclude that 
$$\ell \in\text{Range} (S^*) \iff \liminf\limits_{\alpha \to 0^+} \liminf\limits_{\delta \to 0^+} \langle x^{\delta,\alpha} \, , A^\delta x^{\delta,\alpha} \rangle_{X\times X^*} < \infty.$$
This gives another family of filter functions to use in numerical reconstructions. Also, this simplifies the results in \cite{GLSM} that pertain to the case of a perturbed positive data operator. For more applications of the GLSM, we refer to \cite{glsm-cracks,glsm-stokes,glsm-elastic} for a few examples.

\section{Application to an Inverse Shape Problem in Scattering}\label{recon}
In this section, we will apply the theory developed in Section \ref{new-RegFM} to a problem coming from inverse scattering. Here, we will consider the problem of recovering an isotropic scatterer with a conductive coating from the measured far-field data. The factorization method was initially studied for this problem in \cite{fmconductbc} and another factorization was recently studied in \cite{rafa-fm}. Using the newly derived factorization of the far-field operator derived in \cite{rafa-fm} we will show that Theorem \ref{regfm-thm3} can be applied to this inverse shape problem. We note that, when the given perturbed data operator maps a Hilbert space to itself then the dual-pairing in Theorem \ref{regfm-thm3} is replaced with the inner-product on the Hilbert space.

We let $u$ denote the total field given by $u=u^s+u^i$. Here, the incident plane wave is denoted by  $u^i = \text{e}^{\text{i}kx\cdot \hat{y}}$ with wave number $k>0$ and  incident direction $\hat{y} \in \mathbb{S}^{d-1}$ (i.e. the unit circle/sphere) is used to illuminate the scatterer $D$. Throughout this section, the scatterer $D\subset \mathbb{R}^d$ (with $d=2$ or 3) is a simply connected open set with $C^2$ boundary $\partial D$ with unit outward normal vector $\nu$. When the  incident plane wave  interacts with the scatterer it produces the scattered field $u^s \in H^1_{loc}(\R^d)$ that solves the boundary value problem 
\begin{align} \label{direct1}
\begin{array}{lc}
\Delta u^s+k^2 n(x)u^s= - k^2 \big(n(x) -1 \big) u^i \quad & \text{in} \hspace{.2cm}  \mathbb{R}^d \backslash \partial D\\[1.5ex]
 [\![u^s]\!] =0  \quad \text{ and } \quad  [\![\partial_\nu u^s]\!]  = - \eta(x) \big(u^s+u^i\big)  \quad & \text{on} \hspace{.2cm} \partial D.
\end{array}
\end{align}
Here, the normal derivative is given by  $\partial_\nu \phi = \nu \cdot \nabla \phi$ for any $\phi$. Also, we have that  
 $$[\![ \phi ]\!]  := ( \phi^{+} - \phi^{-})  \quad \text{ and } \quad  [\![\partial_\nu \phi ]\!] := ( \partial_{\nu} \phi^{+} - \partial_{\nu} \phi^{-})$$
with `$-$' and `$+$' corresponds to taking the trace on $\partial D$ from the interior or exterior of $D$, {respectively}. Lastly, to close the system, we impose the Sommerfeld radiation condition on the scattered field $u^s$ given by 
\begin{align}\label{src}
{\partial_r u^s} - \text{i} ku^s =\mathcal{O} \left( \frac{1}{ r^{(d+1)/2} }\right) \quad \text{ as } \quad r=|x| \rightarrow \infty
\end{align}
which holds uniformly with respect to the angular variable $\hat{x}=x/r$. 

We will assume that the refractive index $n \in L^{\infty}(\R^d)$ and conductivity $\eta \in L^{\infty}(\partial D)$. In \cite{fmconductbc}, it has been proven that \eqref{direct1}--\eqref{src} is well-posed provided that 
$$ \Im (n) \geq 0 \,\,  \textrm{  a.e. in $D$}\quad \text{ and } \quad \Im (\eta) \geq 0 \,\,  \textrm{  a.e. on $\partial D$}$$
where supp$(n-1)=D$. 
 Therefore, it is well known that for any incident direction $\hat{y}$ the scattered field $u^s$ has the asymptotic behavior (see for e.g. \cite{TE-book})
$$u^s(x,\hat{y})=\gamma\frac{\text{e}^{\text{i}k|x|}}{|x|^{(d-1)/2}}\left\{u^{\infty}(\hat{x},\hat{y})+\mathcal{O}\left(\frac{1}{|x|}\right)\right\}\hspace{.3cm}\text{as}\hspace{.3cm}|x|\longrightarrow \infty$$
 where the constant $\gamma$ is defined by 
$$ \gamma=\frac{\text{e}^{\text{i}\pi/4}}{\sqrt{8\pi k}}\hspace{.3cm} \text{for} \hspace{.3cm} d=2 \hspace{.3cm}\text{and}\hspace{.3cm}\gamma=\frac{1}{4\pi}\hspace{.3cm}\text{for}\hspace{.3cm} d=3.$$
Here, the far-field pattern $u^{\infty}$ depends on both the observation direction $\hat{x}$ and incident direction $\hat{y}$.  Given the measured far-field pattern we can define the associated far-field operator denoted $F$ which is given by 
\begin{align}\label{ff-operator}
(Fg)(\hat{x})=\int_{\mathbb{S}^{d-1}}u^{\infty}(\hat{x},\hat{y})g(\hat{y})\, \text{d}s(\hat{y}) \quad \text{ for } \quad g\in L^2(\mathbb{S}^{d-1})
\end{align}
mapping $L^2(\mathbb{S}^{d-1})$ into itself.

Now, in order to apply Theorem \ref{regfm-thm3} we need to factorize the far-field operator $F$. To this end, we have that in \cite{rafa-fm} the integral identity 
\begin{align}\label{intrep}
u^s(x) = k^2 \int_D (n(\omega)-1)\Phi(x,\omega)\big(u^s(\omega) &+u^i(\omega) \big) \, \text{d}\omega  \nonumber  \\
&+  \int_{\partial D} \eta(\omega) \Phi(x,\omega)\big(u^s(\omega)+u^i(\omega) \big) \, \text{d}s(\omega)
\end{align}
was proven. In equation \eqref{intrep}, we let $\Phi(x , y)$ denote the radiating fundamental solution for  Helmholtz equation given by 
\begin{align}\label{fund-solu}
\Phi(x , y) = \left\{ 
		\begin{array}{cl}
			\frac{\text{i}}{4} H^{(1)}_{0}(k|x - y|) &\quad \text{for} \quad d = 2, \\[1.5ex]
			\displaystyle \quad\quad  \frac{\text{e}^{\text{i}k|x - y|}}{4\pi|x - y|} &\quad \text{for} \quad d = 3
		\end{array}
	\right.
\end{align}
for $x \neq y$, where $H^{(1)}_{0}$ is the first kind Hankel function of order zero. Notice, that equation \eqref{intrep} corresponds to the Lippman-Schwinger integral equation corresponding to the scattering problem \eqref{direct1}--\eqref{src}.  Using equation \eqref{intrep}, we have the factorization 
$$F= H^* T H$$ 
with 
$$ H: L^2(\mathbb{S}^{d-1}) \longrightarrow L^2(D)\times L^2(\partial D) \quad \text{ given by } \quad Hg = \big(v_g \big|_{D} \, , \, v_g \big|_{\partial D} \big)$$
and it's adjoint $ H^*:L^2(D)\times L^2(\partial D)  \longrightarrow  L^2(\mathbb{S}^{d-1})$ is given by 
$$H^*(\varphi  \, , \, \psi)  = \int_D\text{e}^{-\text{i} k x \cdot \hat{y}}  \varphi(x)\, \text{d}x +  \int_{\partial D}  \text{e}^{-\text{i} k x \cdot \hat{y}} \psi(x)\, \text{d}s(x).$$
Here, $v_g$ is the Herglotz wave operator given by 
$$v_g(x) = \int_{\mathbb{S}^{d-1}} \text{e}^{\text{i} k x \cdot \hat{y}}g(\hat{y})  \, \text{d}s(\hat{y}).$$
The middle operator $T: L^2(D)\times L^2(\partial D) \longrightarrow L^2(D)\times L^2(\partial D)$ is defined by 
$$T(f , h)= \big( k^2(n-1)(f +w)\big|_{D} \,  , \,  \eta (h +w)\big|_{\partial D} \big)$$
for any $(f , h) \in  L^2(D)\times L^2(\partial D)$ where $w \in H^1_{loc}(\R^d)$ is the unique solution to  
\begin{align}\label{w-equ}
\begin{array}{lc}
\Delta w+k^2 n(x)w= - k^2 \big(n(x) -1 \big) f \quad \text{in} \quad \mathbb{R}^d \backslash \partial D\\[1.5ex]
[\![w]\!] =0  \quad \text{ and } \quad  [\![\partial_\nu w]\!]  = - \eta(x) \big(w+h\big)   \quad \text{on} \quad \partial D
\end{array}
\end{align} 
along with the radiation condition \eqref{src}. 

In order to apply Theorem \ref{regfm-thm3}, we need to study the analytical properties of the operator $H$ and $T$. Therefore, notice that $H$ is given by integral operators acting of $L^2(D)$ and $L^2(\partial D)$ with analytic kernels which implies the compactness of $H$. Now, for the injectivity of $H$ we assume that $g\in$ Null$(H)$ which implies that 
$$v_g \big|_{D}=0 \quad \text{ and } \quad  v_g \big|_{\partial D}=0.$$ 
Using that fact that $\Delta v_g +k^2 v_g = 0$ in $\R^d$ we have that unique continuation implies $v_g =0$ in  $\R^d$. Which implies that $g=0$ since the Herglotz wave operator is injective (see for e.g. \cite{coltonkress}). Next, we consider the strict coercivity of the operator $T$. In general, $T$ will not be strictly coercive on the range of $H$. In order to circumvent this issue, we will use the augmented far-field operator 
$$F_{\sharp} = \big|  \Re(F) \big| + \big| \Im (F) \big|$$
where 
$$\Re(F) =\frac{1}{2} (F+F^*) \quad  \text{and} \quad \Im (F) = \frac{1}{2 \text{i}} (F-F^*).$$
Note, that $\Re(F)$ and  $\Im (F) $ are self-adjoint compact operators by definition which implies that the absolute value can be compute via the spectral decomposition i.e. the Hilbert-Schmidt Theorem. Now, we recall the following result from \cite{rafa-fm} pertaining to the analytical properties of the operator $T$.

\begin{lemma}\label{t-thm}
Let $T: L^2(D)\times L^2(\partial D) \longrightarrow L^2(D)\times L^2(\partial D)$ be given by 
$$T(f , h)= \big( k^2(n-1)(f +w)\big|_{D} \,  , \,  \eta (h +w)\big|_{\partial D} \big)$$
where $w \in H^1_{loc}(\R^d)$ is the radiating solution to  \eqref{w-equ}. Then provided that $\Re(n-1)$ and $\Re(\eta)$
are both uniformly positive (or negative) definite we have that: 
\begin{enumerate}
\item $\Re(T)$ is the sum of a coercive operator and compact operator.
\item $\Im(T)$ is positive on the $\overline{\text{Range}(H)}$ when $k$ is not a transmission eigenvalue. 
\end{enumerate}
\end{lemma}

From the results in Lemma \ref{t-thm}, we can conclude that the augmented  far-field operator has the factorization 
$$F_{\sharp} = H^* {T}_\sharp H $$
where the middle operator ${T}_\sharp$ is strictly coercive on  $L^2(D)\times L^2(\partial D)$. This fact is given by the proof of Theorem 2.15 in \cite{kirschbook}. Notice, that in Lemma \ref{t-thm} we must assume that the wave number $k$ is not an associated transmission eigenvalue. These eigenvalues can be seen as wave numbers $k$ for which there exists a non-scattering incident wave. The associated transmission eigenvalue problem for \eqref{direct1}--\eqref{src} has been studied in multiple papers. In \cite{te-cbc} the existence of infinity many eigenvalues was proven for real-valued coefficients and in \cite{te-cbc3} it was proven that the set of transmission eigenvalues is discrete provided that $|n-1|^{-1} \in L^\infty (D)$ and $\eta^{-1} \in L^\infty (\partial D)$, which is the case under our assumptions.

Now, the last piece that we need to prove is for some $\ell_z \in L^2(\mathbb{S}^{d-1})$ depending on a sampling point $z \in \R^d$ we have that $\ell_z \in$ Range$(H^*)$ if and only if $z$ is in the scatterer $D$.  To this end, we let
$$ \ell_z = \text{e}^{- \text{i} k z \cdot \hat{y}}$$
which is the far-field pattern of the fundamental solution $\Phi(z,y)$ defined in \eqref{fund-solu}. With this, we are now ready to connect the Range$(H^*)$ to the scatterer $D$. 

\begin{theorem}\label{h-thm}
Let $ H^*:L^2(D)\times L^2(\partial D)  \longrightarrow  L^2(\mathbb{S}^{d-1})$ is given by 
$$H^*(\varphi  \, , \, \psi)  = \int_D\textnormal{e}^{-\textnormal{i} k x \cdot \hat{y}}  \varphi(x)\, \text{d}x +  \int_{\partial D}  \textnormal{e}^{-\textnormal{i} k x \cdot \hat{y}} \psi(x)\, \text{d}s(x).$$
Then we have that 
$$z \in D \iff  \ell_z \in \textnormal{Range}(H^*).$$
\end{theorem}
\begin{proof}
To prove the claim, we first notice that $H^*(\varphi  \, , \, \psi)=v^\infty$ where the function $v$ is given by 
$$v= \int_D \Phi(\cdot \, ,x)  \varphi(x)\, \text{d}x +  \int_{\partial D} \Phi(\cdot \, ,x) \psi(x)\, \text{d}s(x).$$
Therefore, we have that $v$ is the sum of the volume and single layer potential for Helmholtz equation. 
By the mapping properties of the volume and single layer potential we have that $v  \in H^1_{loc}(\R^d)$. Now, by the jump relation for the  normal derivative of the single layer potential across $\partial D$ (see for e.g. \cite{intequbook}) we have that $v$ is the unique solution to 
\begin{align}\label{v-equ}
\begin{array}{lc}
\Delta v+k^2 v= 0 \quad \text{in} \quad \mathbb{R}^d \backslash \overline{D}  \quad \text{and} \quad \Delta v+k^2 v= - \varphi  \quad \text{in} \quad  D\\ [1.5ex]
[\![ v ]\!] =0  \quad \text{ and } \quad  [\![\partial_\nu v]\!]  = - \psi   \quad \text{on} \quad \partial D
\end{array}
\end{align} 
along with the radiation condition \eqref{src}.

Now, to prove the claim we assume that $z \in D$. Then, we have that $\Phi(z, \cdot )$ is a smooth solution to Helmholtz equation in $\mathbb{R}^d \backslash \{z\}$. This implies that $\Phi(z \, ,\cdot)\big|^{+}_{\partial D} \in H^{3/2}(\partial D)$. By appealing to the Lifting Theorem we have that there is a 
$w_z \in H^2(D)$ where 
$$w^{-}_z= \Phi^{+}(z \, ,\cdot)  \quad \text{on} \quad \partial D.$$
Therefore, we let 
\begin{align*}
v_z = \left\{
\begin{array}{lc}
 w_z  \quad \text{in} \quad  D\\ [1.5ex]
\Phi(z \, , \cdot )  \quad \text{in} \quad \mathbb{R}^d \backslash \overline{D}
\end{array}
\right.
\end{align*} 
and notice that $v_z \in H^1_{loc}(\R^d)$ is a radiating solution to \eqref{v-equ} with 
$$\varphi_z = -  (\Delta w_z+k^2 w_z) \in L^2(D) \quad \text{and} \quad \psi_z =  \big(\partial_{\nu}w^{-}_z - \partial_{\nu} \Phi^{+}(z \, ,\cdot) \big) \in L^2(\partial D).$$ 
Since, $\Phi(z, \cdot ) = v_z$ in $ \mathbb{R}^d \backslash \overline{D}$ we have that $\ell_z = v^{\infty}_z$ which implies that  $\ell_z = H^*(\varphi_z , \psi_z)$.

To proceed by way of contradiction, assume that $z \in \mathbb{R}^d \backslash \overline{D}$ with 
$$\ell_z = H^*(\varphi_z , \psi_z)\quad \text{for some} \quad(\varphi_z , \psi_z) \in L^2(D)\times L^2(\partial D).$$ 
This implies that there exists a radiating solution to \eqref{v-equ} denoted  $v_z \in H^1_{loc}(\R^d)$ such that $\Phi(z\, , \cdot ) = v_z$ in $ \mathbb{R}^d \backslash \big( \overline{D} \cup \{z\} \big)$ by Rellich’s Lemma. By appealing to elliptic regularity (see for e.g. \cite{evans}) we have that $v_z$ is continuous in any ball $B(z ,\epsilon) \subset \mathbb{R}^d \backslash \overline{D}$. Therefore, $v_z$ is bounded near $z$ but $\Phi(z \, , \cdot )$ has a singularity at $z$ which gives a contradiction, proving the claim. 
\end{proof}

With Theorem \ref{h-thm} we have all we need to prove that the regularized factorization method is applicable to our problem. Indeed, by combining the analysis in these section we have the following theorem. This connects the scatterer $D$ to the measured far-field operator that is perturbed by random noise. 

\begin{theorem}\label{recon-thm}
Assume that $F^\delta :  L^2(\mathbb{S}^{d-1}) \longrightarrow  L^2(\mathbb{S}^{d-1})$ is a perturbation of the far-field operator $F$ given by \eqref{ff-operator}  with $F^{\delta}_{\sharp}$ being a positive compact operator.  Provided that $k$ is not a transmission eigenvalue with $\Re(n-1)$ and $\Re(\eta)$ both uniformly positive (or negative) definite, then 
$$z \in D \iff  \liminf\limits_{\alpha \to 0^+} \liminf\limits_{\delta \to 0^+} \big( g_z^{\delta,\alpha} \, , F^{\delta}_{\sharp} g_z^{\delta,\alpha} \big)_{L^2(\mathbb{S}^{d-1})} < \infty$$
where $g_z^{\delta,\alpha}$ is the regularized solution to $F^{\delta}_{\sharp}  g = \ell_z$.
\end{theorem}

In general, one can use the  $\alpha = \alpha(\delta)$ from the method described  in the following section for a known noise level $0<\delta\ll 1$. This will be explored in the following section numerically. Also, one may be able to weaken the assumption on the coefficient $n(x)$ as is done in \cite{GLSM2}.

\section{Numerical Examples}\label{numerics}

Here, we will provide a few numerical examples to illustrate the theoretical results that we have proven in the previous sections. In this section, we will provided numerical reconstructions using  \texttt{MATLAB} R2022a. To this end, we will consider recovering a scatterer $D$ from the far-field pattern $u^\infty$ corresponding to \eqref{direct1}--\eqref{src}. In all our examples, for simplicity we will assume that parameters $n$ and $\eta$ are constants given by 
$$ n=4+2\text{i} \quad \text{and} \quad \eta = 2+\text{i}.$$
Therefore, we have that  \eqref{direct1}--\eqref{src} is well-posed and that Theorem \ref{recon-thm} can be used to recover that scatterer from the measured far-filed operator.

In order to proceed, we need to synthetically compute the far-field pattern. To this end, recall that the scattered field $u^s(x,\hat{y})$ is given by \eqref{intrep} and in \cite{rafa-fm} we see that for scatterers such that $|D| \ll 1$
$$u^{\infty}(\hat{x},\hat{y}) \approx k^2 (n-1) \int_D  \mathrm{e}^{- \mathrm{i} k {\omega}  \cdot ( \hat{x} - \hat{y})} \, \text{d}\omega + \eta \int_{\partial D}  \mathrm{e}^{- \mathrm{i} k {\omega}  \cdot ( \hat{x} - \hat{y})}\, \text{d}s(\omega).$$ 
Note, that we have used the fact that in our examples both $n$ and $\eta$ are constants. This corresponds to version of the Born approximation of the far-field pattern for the scattering problem with a conductive boundary. This implies that the far-field pattern can be computed using a 32 point Gaussian quadrature method in \texttt{MATLAB}. Here, we will assume that the boundary of the scatterer $D$ is given by  
$$\partial D = r(\theta) \left(\cos (\theta), \sin(\theta) \right) \quad \text{ for } \quad 0\leq \theta \leq 2 \pi. $$
In our examples, the $2 \pi$ periodic radial function $r(\theta)$ is given by 
$$ r(\theta) = 0.5\big(1-0.3\sin(4\theta) \big)$$
for a star shaped scatterer. Note, that this is a non-convex scatterer (see Figure \ref{recon1}) and we will see that the regularized factorization method can provide accurate reconstruction even in the presence of noisy data.

In order to discretize that problem, we will compute $u^{\infty}(\hat{x}_i, \hat{y}_j )$ using numerical integration at 64 equally spaced points on the unit circle given by 
$$\hat{x}_i =\hat{ y}_i =(\cos \theta_i , \sin \theta_i) \quad  \text{with}\quad  \theta_i =2\pi(i-1)/64.$$
This gives that discretized far-field operator given by 
$${\bf F} = \left[u^{\infty}(\hat{x}_i, \hat{y}_j ) \right]_{i,j=1}^{64}$$
will be used to recover the scatterer. Therefore, just as in \cite{harris1} we have that the imaging functional that discretizes version of Theorem \ref{recon-thm} is to plot the imaging functional 
\begin{align}\label{imag-func}
W(z)= \left[ \sum\limits_{j=1}^{64} \frac{\phi^2(\sigma_j  ; \alpha)}{\sigma_j} \big|({\bf u}_j , \boldsymbol{ \ell}_z)\big|^2 \right]^{-1} \quad \text{ with } \quad  \boldsymbol{\ell}_z = [\text{e}^{-\text{i}k \hat{x}_i \cdot z}]_{i=1}^{64}.
\end{align}
Here $\sigma_j$ are the singular values and ${\bf u}_j$ are the left singular vectors of 
$${\bf F}_{\sharp} = \big|  \Re({\bf F}) \big| +  \big| \Im ({\bf F}) \big|$$ 
and the filter function $\phi(t ; \alpha)$ is given by \eqref{filters} or \eqref{GLSMfilter}. Note, that the absolute value of a self-adjoint matrix is given by it's eigenvalue decomposition. \\

{\bf Example 1:} In our first reconstruction in Figure \ref{recon1}, we assume that we have the  discretized far-field operator with no noise added to the data. Then, we can plot the imaging functional $W(z)$ using the Landweber filter function given in \eqref{filters} with parameters $\alpha = 10^{-5}$ and $\beta = 1/(2\| {\bf F_\sharp} \|_2^2)$. From this we see that the reconstruction with and without regularization gives good reconstructions of the scatterer. 
\begin{figure}[ht]
\centering 
\includegraphics[scale=0.28]{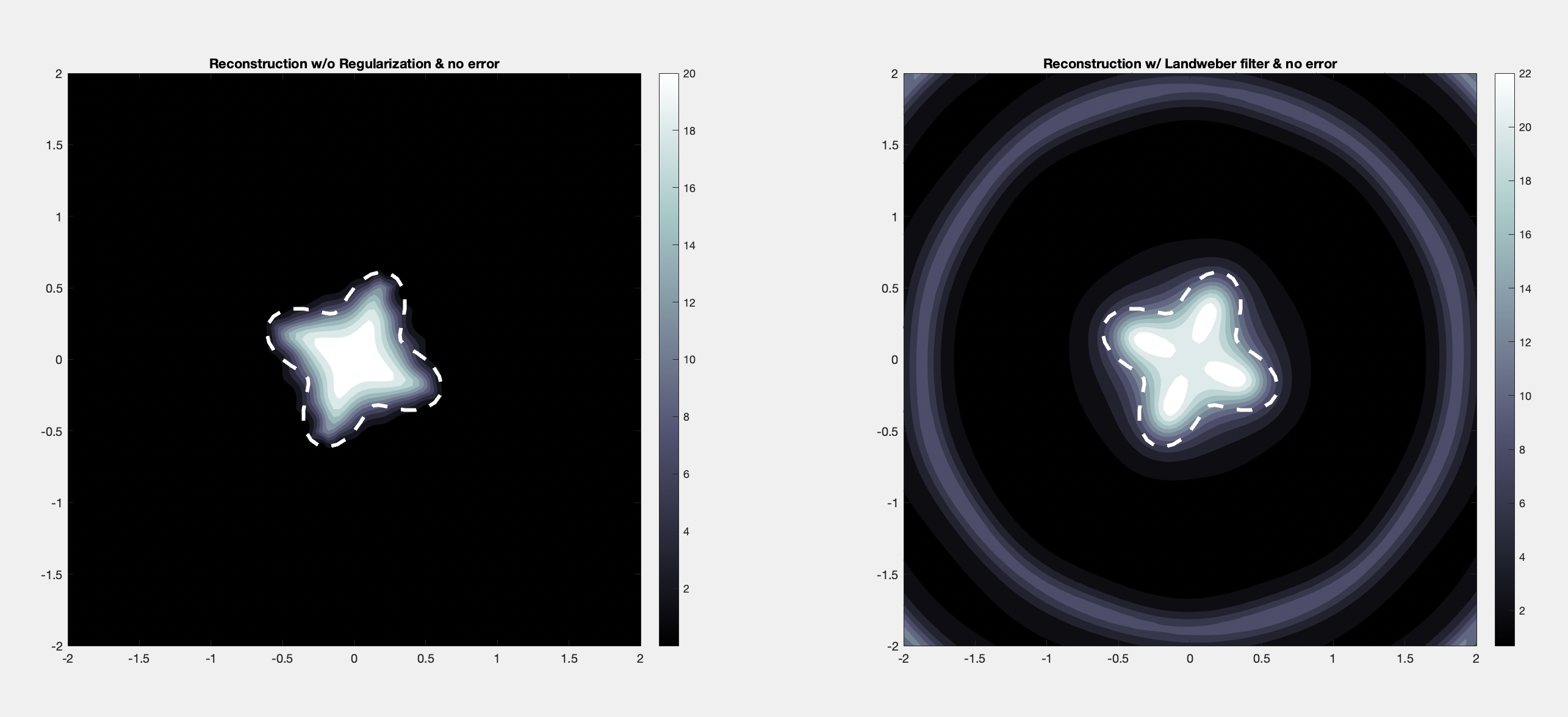}
\caption{Reconstruction of the scatterer with and without regularization where no error is added to the far-field data. Here we us the Landweber filter given in \eqref{filters}. Left: reconstruction  without regularization and Right: reconstruction with regularization.}
\label{recon1}
\end{figure}\\

Now, we wish to show that when there is added noise in the data that regularization is required for reconstructing the scatterer. To this end, we need to define the discretized far-field operator with random noise added which is given by 
$${\bf F}^{\delta} = \left[u^{\infty}(\hat{x}_i, \hat{y}_j ) \left( 1 +\delta E_{i,j} \right) \right]_{i,j=1}^{64}$$
with random complex-valued matrix $\mathbf{E}$ satisfying $\| \mathbf{E} \|_2 =1$. Again, the far-field pattern $u^{\infty}(\hat{x}_i, \hat{y}_j )$ is again computed via the numerical integration as in Figure \ref{recon1}. Here, the real and imaginary parts of the matrix $\mathbf{E}$ are randomly distributed between $\pm 1$ and then normalized. In this case, we let 
$${\bf F}^\delta_{\sharp} = \big|  \Re({\bf F}^{\delta}) \big| +  \big| \Im ({\bf F}^{\delta}) \big|$$ 
and in \eqref{imag-func} we use the singular values and vectors corresponding to the operator ${\bf F}^\delta_{\sharp}$. 
In the following examples, we see how noise added to the far-field data affects the reconstruction with and without regularization.  \\

{\bf Example 2:} In the reconstructions given in Figure \ref{recon2}, we present the case with error added to the data. Just as in the previous example, we use Landweber filter function for our regularization scheme where again we take $\beta = 1/(2\| {\bf F_\sharp} \|_2^2)$. We see that the added noise in the data corrupts the reconstruction without regularization. Here, we let the noise level $\delta =0.05$ and take  $\alpha = 10^{-5}$ ah-hoc as in the previous case.  \\
\begin{figure}[ht]
\centering 
\includegraphics[scale=0.28]{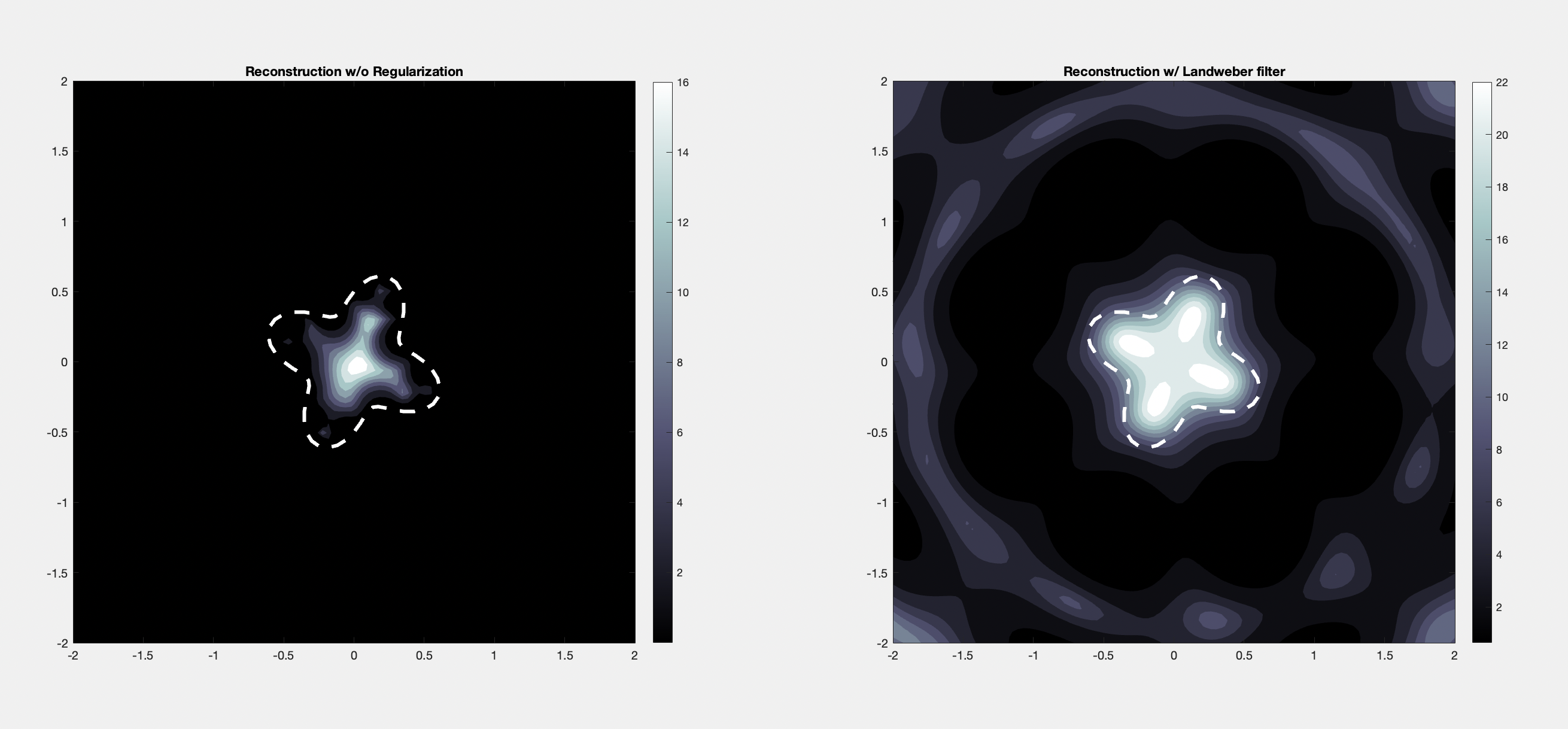}
\caption{Reconstruction of the scatterer with and without regularization where $5\%$ error is added to the far-field data. Here we us the Landweber filter given in \eqref{filters}. Left: reconstruction  without regularization and Right: reconstruction with regularization.}
\label{recon2}
\end{figure}\\

{\bf Example 3:} In the reconstructions given in Figure \ref{recon3}, we give the reconstruction using the Tiknohov filter given in \eqref{filters}. Again, we wish to show that the regularization stabilizes the reconstruction. To this end, we again provide a numerical reconstruction of the scatterer $D$ with and without regularization. In this example, we again let the given noise level $\delta =0.05$ and take  $\alpha = 10^{-5}$ ah-hoc. \\
\begin{figure}[H]
\centering 
\includegraphics[scale=0.28]{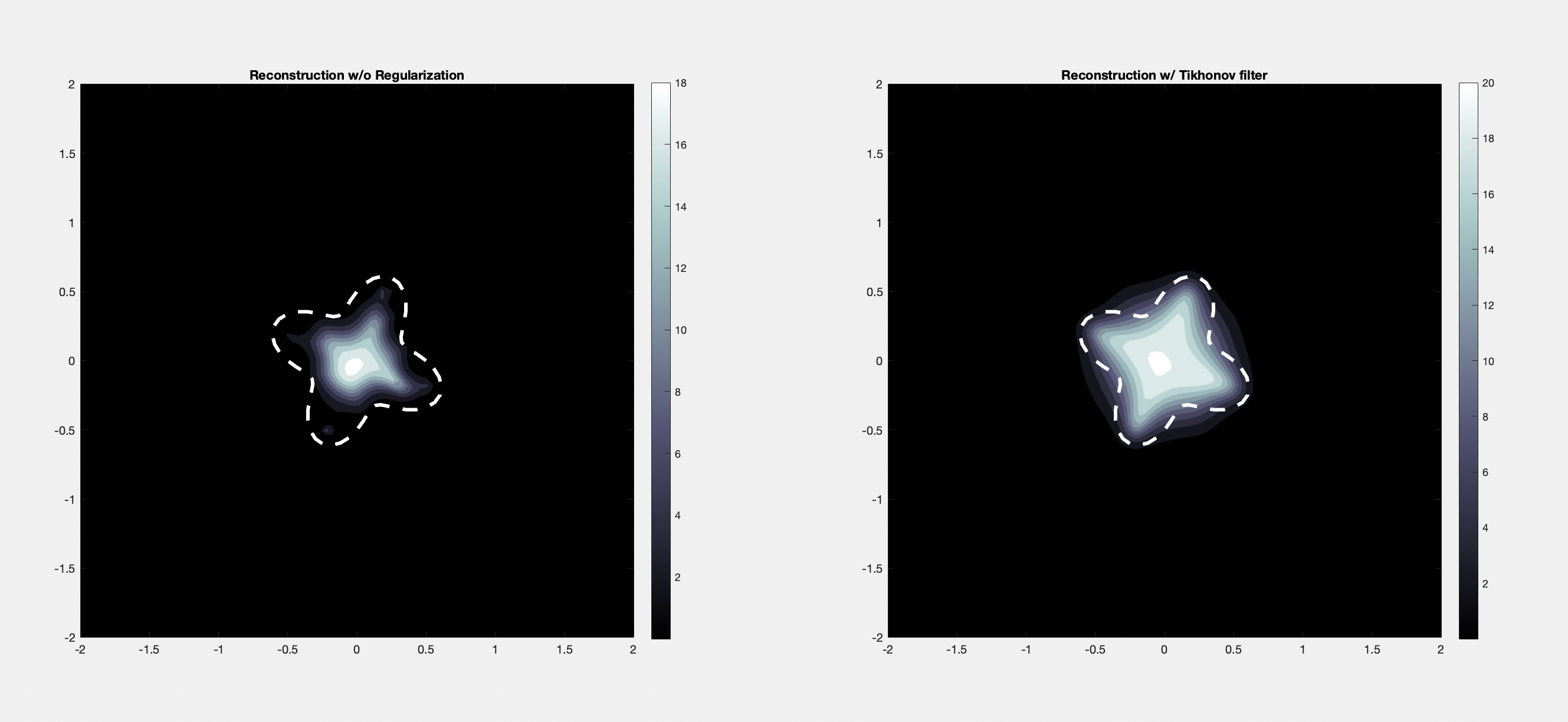}
\caption{Reconstruction of the scatterer with and without regularization where $5\%$ error is added to the far-field data. Here we us the Tiknohov filter given in \eqref{filters}. Left: reconstruction  without regularization and Right: reconstruction with regularization.}
\label{recon3}
\end{figure}

{\bf Example 4:} In the reconstructions given in Figure \ref{recon4}, we yet again show that the regularization helps provide stability with noisy data. For this example, we use the filter function associated with the GLSM given by \eqref{GLSMfilter}. Again we can see that the regularization provides needed stability with respect to noisy data.  Here we take the noise level $\delta =0.05$ and the regularization parameter  $\alpha = 10^{-5}$.\\
\begin{figure}[ht]
\centering 
\includegraphics[scale=0.28]{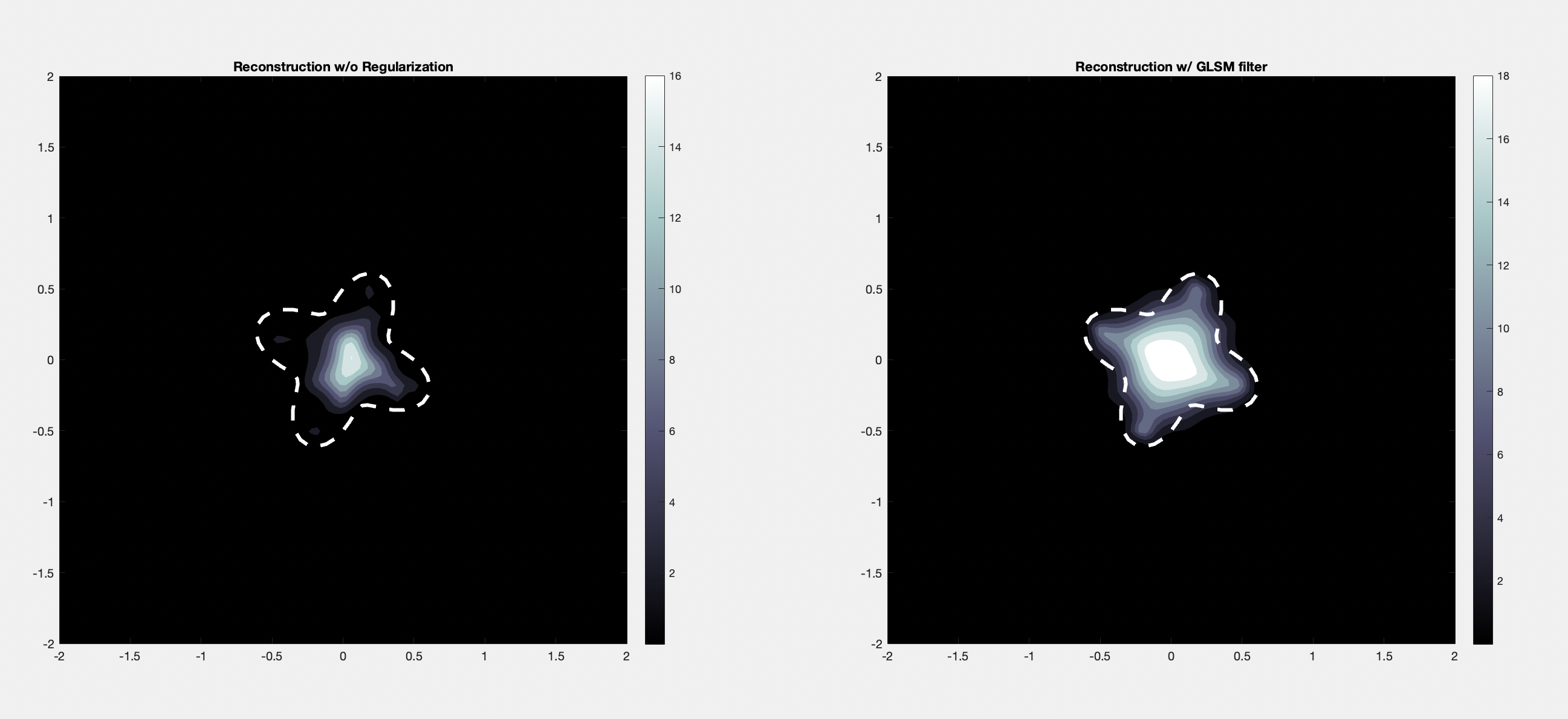}
\caption{Reconstruction of the scatterer with and without regularization where $5\%$ error is added to the far-field data. Here we us the GLSM filter given in \eqref{GLSMfilter}. Left: reconstruction  without regularization and Right: reconstruction with regularization.}
\label{recon4}
\end{figure}\\

Now, we are interested in determining the regularization parameter $\alpha = \alpha(\delta)$ via analytical means motivated by the proof of Theorem \ref{regfm-thm3}. To this end, we notice that for $\alpha(\delta)$ in the inequality \eqref{reg-ineq} we would require 
$$ C^2_{\alpha(\delta)} \sqrt[4]{\delta}  \longrightarrow 0^+ \quad \text{ and } \quad \alpha(\delta) \longrightarrow 0^+  \quad \text{ as } \quad \delta \longrightarrow 0^+.$$
Here $C_{\alpha}$ depends on which regularization filter function is used.  Therefore, in order to determine a suitable $\alpha(\delta)$ we will solve $C^2_{\alpha(\delta)} \sqrt[4]{\delta}  = \delta^p$ for some $p>0$. From this we obtain the regularization parameters
\begin{align} \label{reg-param}
\alpha_{\text{Tik}}(\delta) =  \frac{1}{4} \delta^{\left(\frac{1}{4}-p\right)}, \quad  \alpha_{\text{Land}}(\delta) =  \frac{1}{2 \| {\bf F}^\delta_{\sharp}  \|_2^2}  \delta^{\left(\frac{1}{4}-p\right)}   \quad \text{ and } \quad\alpha_{\text{GLSM}}(\delta) =  \delta^{\frac{1}{2}\left(\frac{1}{4}-p\right)}
\end{align}
for Tikhonov regularization, Landweber iteration and the GLSM, respectively. Note, that for the Landweber iteration we have taken $\beta = 1/(2\| {\bf F^\delta_\sharp} \|_2^2)$ as in the previous examples. From the fact that we require  $\alpha(\delta) \longrightarrow 0^+$ as  $\delta \longrightarrow 0^+$, this implies that $p \in (0 , 1/4)$. Also, we take the $m =\lceil 1/ \alpha_{\text{Land}}(\delta)\rceil$ to be the parameter in the Landweber iteration.\\

{\bf Example 5:} In the reconstructions given in Figure \ref{recon5}, we test the regularization parameter $\alpha (\delta)$ given by \eqref{reg-param}. We present the numerical reconstruction of the scatterer where we pick $p=1/8$ for each of filter function. This gives that 
$$\alpha_{\text{Tik}}(\delta) =  \frac{1}{4} \delta^{1/8}, \quad  \alpha_{\text{Land}}(\delta) =  \frac{1}{2 \| {\bf F}^\delta_{\sharp}  \|^2}  \delta^{1/8}   \quad \text{ and } \quad\alpha_{\text{GLSM}}(\delta) =  \delta^{1/16}$$
as the given regularization parameter. Here we let the given noise level $\delta =0.01$ and present the reconstructions by each regularizing filter with it's associated regularization parameter. For Figure \ref{recon5}, we compute the regularization parameters  
$$\alpha_{\text{Tik}}(0.01) = 0.1406, \quad  \alpha_{\text{Land}}(0.01) = 3.0784\times10^{-8}  \quad \text{ and } \quad\alpha_{\text{GLSM}}(0.01) =  0.7499$$
which are used in the reconstruction. 

\begin{figure}[H]
\centering 
\includegraphics[scale=0.33]{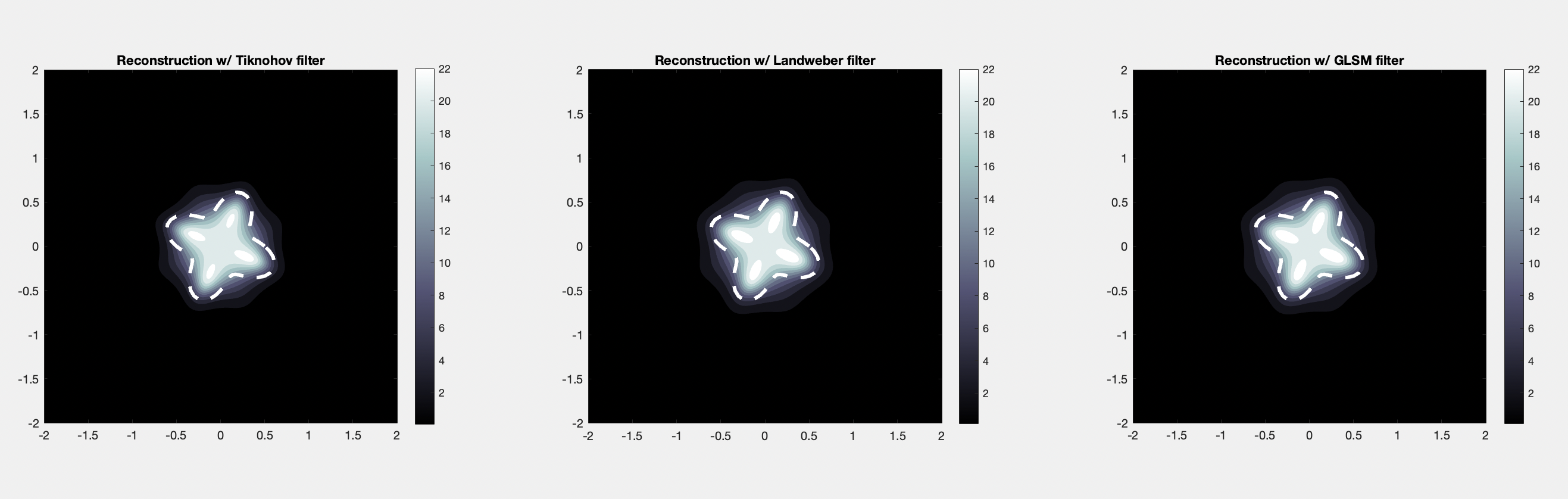}
\caption{The reconstruction via Tiknohov, Landweber and GLSM filters given in \eqref{filters} and \eqref{GLSMfilter}. Here the regularization parameter $\alpha(\delta)$ is given  by \eqref{reg-param} for $p=1/8$.}
\label{recon5}
\end{figure}

\section{Conclusion}
In conclusion, we have given an analytical and numerical study of the regularized factorization method with noisy data. In this paper, we have proven that the regularization strategy discussed here  is stable with respect to noise as well as computationally simple to implement. Indeed, in order to implement this method we see that one only needs the singular value decomposition  of the data operator to provided stable reconstructions. We have also given an analytical method for determining a suitable regularization parameter. For an application of this method, we have applied the regularized factorization method to an inverse scattering problem for recovering a scatterer from the far-field data. As in \cite{GH1,harris1} we know that this method can be applied to other imaging modalities such as electrical and diffuse optical tomography. \\

\noindent{\bf Acknowledgments:} The research of I. Harris is partially supported by the NSF DMS Grant 2107891.


\end{document}